\documentclass[a4paper,12pt,reqno]{amsart}

\title{On Seshadri constants of non-simple abelian varieties}

\author{Rikito Ohta}
\address{Department of Mathematics,
Graduate School of Science,
Osaka University,
Machikaneyama 1-1,
Toyonaka,
Osaka,
560-0043,
Japan.
}
\email{r-ohta@cr.math.sci.osaka-u.ac.jp}

\date{\today}

\usepackage{mymacros,fullpage}
\usepackage{amscd}
\usepackage[all]{xy}

\newcommand{\mult}{\operatorname{mult}}

\begin{document}
\maketitle

\begin{abstract}
Let $(A,L)$ be a polarized abelian variety of dimension $n$.
We prove two results for the Seshadri constants of abelian varieties.
The first one is that if the Seshadri constant $\epsilon(A,L)$ is relatively small with respect to $L^n$,
there exists a proper abelian subvariety $B$ such that $\epsilon(A,L)$ is equal to 
$\epsilon(B,L|_B)$.
Second, we prove that
if there exists a codimension one abelian subvariety $D$ satisfying certain numerical conditions, then $\epsilon(D,L|_D)$ is equal to $\epsilon(A,L)$.
As an application of these results, we investigate the structure of low dimensional polarized abelian varieties whose Seshadri constants is sufficiently small with respect to $L^n$.

 
\end{abstract}


%

\section{Introduction}\label{sc:intro}
Let $X$ be a projective variety of dimension $n$ over $\bC$, the field of complex numbers.
For an ample line bundle $L$ on $X$ and a point $x \in X$, 
Demailly introduced an invariant called the \emph{Seshadri constant} in \cite{MR1178721} as
\begin{align}\label{eq:seshadri_constant}
\varepsilon(X,L;x) \coloneqq 
\inf_{x \in C} \{ \varepsilon_{C,x}(L) \},
\end{align}
where we denote
\begin{align}
\varepsilon_{C,x}(L) \coloneqq \frac {L.C}{\mult_x C}
\end{align}
for a curve $C$ containing $x$.
Note that the definition immediately implies that $\varepsilon(X,L;x)$ is determined by the numerical class of $L$.

It is closely related to various geometric notions such as separation of jets, very ampleness of adjoint bundles, and the symplectic packing problem. We recommend \cite[Chapter 5]{MR2095471} for details.
%
For a polarized abelian variety $(A,L)$,
the Seshadri constant $\varepsilon(A,L;x)$ does not depend on the choice of a point $x \in A$, and hence we denote it by $\varepsilon(A,L)$.
Moreover, we often denote $\varepsilon_{C,0}(L)$ by $\varepsilon_{C}(L)$.
An important problem is to understand how $\varepsilon(A,L)$ reflects the global structure of
$(A,L)$.

For example, 
it is known that there exists a close relationship
between the Seshadri constants and minimal period lengths (also called the Bauer-Sarnak invariants) of polarized complex tori.
For details, see \cite{MR1406008}, \cite{MR1660259} or \cite[Chapter 5]{MR2095471}.
Apart from this, Nakamaye \cite{MR1393263} proved the very interesting following result.

\begin{theorem}[{$=$\cite[Theorem $1.1$]{MR1393263}}]\label{th:nakamaye}
Let $(A,L)$ be a polarized abelian variety of dimension $n$.
Then  $ \varepsilon(A,L) \ge 1 $.
 Moreover,  $ \varepsilon(A,L) =1$ if and only if 
 $(A,L)$ is isomorphic to $(E, L_1)  \times (B, L_2)$,
where $L_1$ is a line bundle of degree 1 on an elliptic curve $E$ and
$(B,L_2)$ is a polarized abelian variety of dimension $n-1$.
\end{theorem}

In the same paper, Nakamaye also proved \cite[Lemma $3.3$]{MR1393263}, which says that
if there exists a curve on $A$ such that 
\begin{align}
\varepsilon_{C} (L) < \frac{\sqrt[n]{L^n}}{n},
\end{align}
then $C$ is a contained in a proper abelian subvariety of $A$.



In view of \cite[Lemma $3.3$]{MR1393263}, it is natural to ask wether
there exists a proper abelian subvariety $B$
which computes the Seshadri constant (i.e., $ \varepsilon(A,L) = \varepsilon(B,L|_B)$) under the assumption of \cite[Lemma $3.3$]{MR1393263}.
Some abelian varieties admit infinitely many abelian subvarieties. If A is one of them, then for a sequence of curves
$\{0 \in C_n\}_n$ 
such that 
$ \{ \varepsilon_{C_n} (L) \}_n $ 
converges to $\varepsilon(A,L)$, it is not clear whether we can take a subsequence such that all $C_n$ are contained in the same proper abelian subvariety of $A$.
Our first result gives an affirmative answer to this question.

\begin{theorem}\label{th:low_seshadri}
Assume that 
\begin{align}\label{eq:main_assumption}
\varepsilon (A,L) < \frac{\sqrt[n]{L^n}}{n}.
\end{align}
Then there exists a proper abelian subvariety $B$ of $A$ such that 
$ \varepsilon(A,L) = \varepsilon(B,L|_B)$.
\end{theorem}

We prove \pref{th:low_seshadri} by combining \cite[Lemma $3.3$]{MR1393263} and 
some finiteness theorems for abelian varieties.

Apart from the above, we also prove the following theorem.
\begin{theorem}\label{th:main2}
Let $(A,L)$ be a polarized abelian variety of dimension $n$.
Fix a positive real number $a$.
Let
$D$ be an abelian divisor in $A$
such that
\begin{align}\label{eq:introassumption}
\sqrt[n]{L^n} \ge \sqrt[n-1]{a (L|_D)^{n-1} } .
\end{align}
If 
$\varepsilon (A,L) < a \sqrt[n]{L^n} / n $ holds,
then
$\varepsilon (A,L)$ = $\varepsilon (D,L|_D)$.
Moreover, if one can take $a \ge (\sqrt[n]{n})^{n-1}$, the upper bound
$\varepsilon (A,L) < a \sqrt[n]{L^n} / n $ automatically holds.
\end{theorem}




We prove \pref{th:low_seshadri} and \pref{th:main2} in \pref{sc:proof}.

We discuss various applications of \pref{th:low_seshadri} and \pref{th:main2} in
\pref{sc:applications}.  We first show that the proof of \pref{th:main2} induces an interesting relationship between the set of Seshadri curves and the set of abelian divisors satisfying \pref{eq:introassumption}
(see \pref{pr:main2} and \pref{eq:containing} for details).
For abelian surfaces, we obtain the following corollaries.

\begin{corollary}[{$ =$ \pref{pr:submaximal2}}]
Let $(S,L)$ be a polarized abelian surface.
Assume that there exists a curve $C \ni 0$ such that
\begin{align}\label{eq:intro2dim}
  \varepsilon_{C} (S,L) < {\sqrt{\frac {L^2}{2}}}.
\end{align}
Then $C$ is elliptic, and it is the unique curve satisfying \pref{eq:intro2dim} and containing $0 \in S$.
\end{corollary}

\pref{th:low_seshadri} and \pref{th:main2} mentioned above mean that the computation of 
the Seshadri constant of an abelian variety can be reduced to that of its abelian subvariety in some cases.
There exist many results which compute the Seshadri constants on 
concrete low-dimensional abelian varieties.
\cite{MR1612681} and \cite{MR1974680} (respectively, \cite{MR1810121} and \cite{MR2112583}) handle the case of the Theta divisors on Jacobian varieties of curves
(resp. principally polarized abelian varieties).
The Seshadri constants of abelian surfaces have been studied in further detail.
For example, it is known that the Seshadri constants of abelian surfaces are rational
(for more results, see Appendix of \cite{MR1660259}, \cite{MR1660260}, \cite{intsehadri}, etc). In the latter part of \pref{sc:applications}, keeping these previous works in mind, we give some elementary applications of our theorems.
In particular, we show the following theorem.

\begin{corollary}\label{cr:main3}
Assume
$L^3 \le 174$ 
and
$  \varepsilon (A,L) < \sqrt[3]{L^3} /3 $.
Then $ \varepsilon (A,L)= 1$ or $4/3$.
Moreover, if $\varepsilon (A,L) = 4/3$, 
$A$ contains the Jacobian variety $J$ of a genus two curve such that
any curve satisfying $\varepsilon_C (L) < 21 \sqrt[3]{L^3} / 8$ is contained in $J$
and $A \simeq J \times E$ for some elliptic curve.
If $ L^3 \le 60$, we obtain $\varepsilon (A,L)= 1$.
\end{corollary}
Note that the assumption  $L^3 \le 60$ is optimal for $\varepsilon (A,L)= 1 $.
In fact, for any $n \in 6 \bZ $ satisfying $n > 60$,
we can construct examples of $(A,L)$ satisfying $\varepsilon(A,L) < \sqrt[3]{L^3}/ 3$, $L^3 =n$, and $\varepsilon(A,L) \neq 1$ (see \pref{eg:3-dim} for this).

\subsection*{Acknowledgements}
The author would like to appreciate his advisor Shinnosuke Okawa for a lot of useful comments and warm encouragement.

\subsection*{Notation and conventions}
The ground field is $\bC$.
A \emph{polarized abelian variety} is a pair $(A,L)$ of an abelian variety $A$ (i.e., smooth projective group scheme over $\bC$) and an ample line bundle $L$ on $A$. 
The identity of $A$ is denoted by $0 \in A$.
An irreducible closed subvariety $B \subset A$ is an \emph{abelian subvariety} of $A$ if $B$ is a group subscheme of $A$ by the inclusion.
An abelian subvariety of $A$ of codimension one is called an \emph{abelian divisor} of $A$.
A \emph{curve} is a projective and integral scheme over $\bC$ of dimension one.
We say a curve in $A$ generates an abelian subvariety $B$ if $B$ is the minimal abelian 
subvariety containing the curve.

\section{Proof of the theorems}\label{sc:proof}
In this section, we prove
\pref{th:low_seshadri} and \pref{th:main2}.

\subsection{Proof of \pref{th:low_seshadri}}
The key ingredient of the proof is \pref{lm:finiteness}, which asserts the existence of the minimal element in 
the set of the Seshadri constants of polarized abelian subvarieties of $(A,L)$  of bounded degree.
We begin with some preparation.
\begin{definition}\label{df:quotient_set}
Let $A$ be an abelian variety of dimension $n$.
We define the following sets.

\begin{enumerate}

\item
For a fixed abelian variety $B$ of dimension $k$,
\begin{align}\label{eq:def_S}
S_{B,r} 
\coloneqq \{ L' \mid \text{ $L'$ is an ample line bundle on B satisfying ${L'} ^k < r$ }\} / \sim,
\end{align}
where $L_1 \sim L_2$ if there exists an automorphism $f$ of $B$ such that
$[f^* L_1] = [L_2] $ in $\operatorname{NS} (B)$.

\item
For an ample line bundle $L$ on $A$,
\begin{align}\label{eq:finite_set}
{S}_{k,r}^L \coloneqq
\{B \mid  \text{$B$ is an abelian subvariety of $A$ of dimension $k$ such that 
$({L|_B)}^k  < r$ } \}.
\end{align}

\end{enumerate}

Consider the following maps.
\begin{align}
\varepsilon \colon {S}_{k,r}^L \to \bR_{\ge 0};  \ B \mapsto \varepsilon (B, L|_B).
\end{align}
Then we define the following sets.
\begin{align}\label{eq:kakiku}
E_{k,r}^L \coloneqq \operatorname{Im}(\varepsilon)
= \{\varepsilon(B,L|_B) \in \bR_{>0} \mid  B \in {S}_{k,r}^L \}.
\end{align}

\end{definition}

\begin{lemma}\label{lm:finiteness}
Let $L$ be an ample line bundle on $A$.
Then
$E_{k,r}^L $ is a finite set for any $1 \le k \le n $ and $r$.
\end{lemma}

\begin{proof}

Assume ${S}_{k,r}^L \neq \emptyset $. By \cite[Theorem]{MR1378542}, there exist only finitely many isomorphism classes of abelian subvarieties of $A$ of dimension $k$.
Let $B_1, B_2 \dots , B_t$ be representatives.
Then we obtain the following map.
\begin{align}
\alpha \colon S_{k,r} ^ L  \to \bigcup_{i=1} ^ {t} S_{B_i,r} ; \ \ 
B \mapsto [\varphi^* _{B} (L|_{B})],
\end{align}
where $\varphi _ B \colon B_i \to B $ is an isomorphism for some $B_i$.
By the definition of $S_{B_i,r}$ (see \pref{eq:def_S}),
it follows that
the following map \pref{eq:additional map} is well-defined.
\begin{align}\label{eq:additional map}
\tilde{\varepsilon}: \bigcup_{i=1} ^ {t} S_{B_i,r} \to \bR_{\ge 0};
\ \  [M] \in S_{B_i,r} \mapsto \varepsilon(B_i, M).
\end{align}
Then we can see that $\tilde{\varepsilon} \circ \alpha = \varepsilon$
since 
\begin{align}
\tilde{\varepsilon} \circ \alpha (B) 
= \varepsilon (B_i, \varphi_B ^* (L|_B))
= \varepsilon (B, L|_B)
= \varepsilon (B).
\end{align}
Hence it is sufficient to show that $S_{B,r}$ is a finite set for a fixed $k$-dimensional abelian variety $B$ to prove $E_{k,r}^L  = \varepsilon({S}_{k,r}^L)$ is finite.

However, this follows from the geometric finiteness theorem (for example, see \cite[Theorem 18.1]{MR861974}), which says that there exist only finitely many classes of ample line bundles of fixed degree in $ \operatorname{NS} (A)$ up to the action of the group of automorphisms of $A$.
\end{proof}

%

Now we are ready to prove \pref{th:low_seshadri}.

\begin{proof}[Proof of \pref{th:low_seshadri}]

By [Nak96,Lemma3.3] and the assumption \pref{eq:main_assumption}, $A$ has a proper abelian subvariety. 
Let $k$ be the maximal dimension of the proper abelian subvarieties of $A$.
For each natural number $1 \le i \le k $, 
we write
\begin{align}
r_i \coloneqq \lb \frac{i \sqrt[n]{L^n}}{n}\rb ^i .
\end{align}
Let 
\begin{align}
 a_{k+1} \coloneqq \frac{\sqrt[n]{L^n}}{n}.
\end{align}
For each $1 \le i \le k$,
starting with $i=k$,
we define $a_i$ inductively as
$
 a_{i} \coloneqq \min \{ \min (E_ {i,r _ {i}} ^ L ), a_{i+1} \}
$,
where $E_{k,r_i}^L$ is defined in \pref{eq:kakiku}.
Obviously, the definition of $a_i$ implies
\begin{align}\label{eq:seq}
 \varepsilon(A,L) \le a_1 \le a_2 \le \cdots \le a_k \le a_{k+1} = \frac{\sqrt[n]{L^n}}{n}.
\end{align}
Now for the proof of the theorem, consider the following conditions for each $i$.
\begin{enumerate}
\item[$(1^{i})$]
$\varepsilon(A,L) < a_i$, and
any curve $C$ satisfying
$
\varepsilon(A,L) \le \varepsilon_{C} (L)
< a_{i} 
$
generates an abelian subvariety of dimension at most $(i-1)$.

\item[$(2^i)$]
$\varepsilon(A,L) = \varepsilon(B,L|_B)$ for some $i$-dimensional abelian subvariety $B$.
\end{enumerate}
Note that our assumption \pref{eq:main_assumption} and \cite[Lemma $3.3$]{MR1393263} imply $(1^{k+1})$.
To prove the theorem, it is sufficient to show that there exists some $i$, for which the condition $(2^i)$ holds.
However, this follows from the following \pref{cl:induction}.
\end{proof}
\begin{claim}\label{cl:induction}
Under the above notation,
$(1^i)$ implies either $(1^{i-1})$ or $(2^{i-1})$ for any $2 \le i \le k+1 $.
\end{claim}

\begin{proof}[Proof of the Claim]
Assume that $(1^{i})$ holds.
Note that it follows that
$\varepsilon(A,L) \le \min (E_ {{i-1},r _ {i-1}}^L)$
by the definition of the Seshadri constant.
The equality implies the condition $(2^{i-1})$,
so let us assume that the inequality is strict.
In this case there exists a curve $C$ satisfying 
\begin{align}\label{eq:ttt}
\varepsilon(A,L) \le \varepsilon_{C} (L) < a_{i-1}.
\end{align}
If all the curves satisfying \pref{eq:ttt} generate abelian subvarieties of 
dimension at most $i-2$, then this implies $(1^{i-1})$, so that the proof is done.
Note that we already know that $\dim B \le i-1$ by $(1^i)$.
Hence, for the contradiction, suppose that there exists a curve $C$ which satisfies \pref{eq:ttt} and generates an abelian subvariety $B$ of dimension $i-1$.

First assume that 
\begin{align}
 \frac{\sqrt[n]{L^n}}{n} \le \frac {\sqrt[i-1]{(L|_B) ^{i-1}}} {i-1}.
\end{align}
In this case, we obtain that 
\begin{align}
\varepsilon_{C}(L) < \frac{\sqrt[n]{L^n}}{n} \le \frac {\sqrt[i-1]{(L|_B) ^{i-1}}} {i-1},
\end{align}
where the first inequality follows from \pref{eq:seq}.
Then \cite[Lemma $3.3$]{MR1393263} implies that $C$ is contained in a proper abelian subvariety of $B$. However, this contradicts that $C$ generates $B$.

Hence, let us assume that 
\begin{align}\label{eq:rr}
\frac{\sqrt[n]{L^n}}{n} > \frac {\sqrt[i-1]{(L|_B) ^{i-1}}} {i-1}.
\end{align}
Note that \pref{eq:rr} holds if and only if $B \in {S}_{i-1,r_{i-1}}^L$ by the definition \pref{eq:finite_set}.
Then we obtain the inequality 
\begin{align}
  \min ({E}_{i-1,r_{i-1}} ^ L) \le \varepsilon_{C}(L).
\end{align}
However, this contradicts our assumption 
$\varepsilon_{C}(L) < a_{i-1} \le \min ({E}_{i-1,r_{i-1}} ^ L)$.
Hence $C$ can not generate an $i-1$ dimensional abelian subvariety and this concludes the proof.
\end{proof}

\subsection{Proof of \pref{th:main2}}
In this subsection, 
we give the proof of \pref{th:main2}.
First, we define the nef threshold of a divisor $D$ on an $n$-dimensional polarized abelian variety $(A,L)$ as
\begin{align}
\sigma(L,D) \coloneqq \sup \{t\in \bR \mid  \text{$L-tD$ is ample}\} \in 
\bR_{>0} \cup \{ \infty \}.
\end{align}
The following lemma is crucial for the proof of \pref{th:main2}.
\begin{lemma}\label{lm:ampleness}
Let $D$ be an abelian divisor of $A$.
Then $ \sigma(L,D) (L |_D) ^{n-1} = L^n /n $.
\end{lemma}

\begin{proof}
For any $x \in A \setminus D$,
it follows that
$(D+x) \cap (D) = \phi$.
Hence we obtain 
$D^2 = 0 $ in the Chow ring $A^2 (A)$
since $(D+x)$ and $D$ are in the same numerically class.
By \cite[Proposition 1.1]{MR2390296}, 
$\sigma(L,D)$ is the multiplicative inverse of the maximal root of the polynomial
\begin{align}
\chi (uL-M) = \frac{1}{n!} (uL - D) ^ n
=\frac{1}{n!} (L^n u^n - n(L|_D )^{n-1} u^{n -1})
\in \bQ [ u ]  .
\end{align}
Then the straightforward computation implies the assertion.
\end{proof}

\begin{remark}
We can also prove this lemma by applying the methods of the Okounkov body.
For details, see the proof of \cite[Corollary 4.12]{Lozovanu:aa}. 

\end{remark}

Now we obtain \pref{th:main2} from \pref{pr:main2} and \pref{lm:main2} below.

\begin{proposition}\label{pr:main2}
Let 
$ D $
be an abelian divisor of $A$.
Suppose that 
\begin{align}\label{eq:main2prop}
  \sqrt[n]{L^n} > (\text{respectively,}  \ge) \sqrt[n-1]{a (L|_D)^{n-1} } 
\end{align}
for a positive real number $a$.
Then any curve $C$ satisfying
\begin{align}\label{eq:aaa}
\varepsilon_{C} (L) \le (resp. <) \frac{a \sqrt[n]{L^n}}{n}
\end{align}
is contained in $D$.
\end{proposition}

\begin{proof}
By \pref{lm:ampleness} and the assumption \pref{eq:main2prop}, we obtain that
\begin{align}
\sigma(L,D) > (\text{resp}. \ge)  \frac{a \sqrt[n]{L^n}}{n}.
\end{align}
This implies that 
$L - \frac{(a \sqrt[n] {L^n})}{n} D $
is ample (resp. nef). Hence it follows that
\begin{align}\label{eq:ppp}
L.C > (\text{resp}. \ge)   \frac {a \sqrt[n]{L^n} (D.C)}{n}.
\end{align}
Then, by the assumption \pref{eq:aaa} and \pref{eq:ppp}, we have
\begin{align}\label{eq:xx}
\frac{a\sqrt[n]{L^n}}{n}   \ge (\text{resp}. >)   \frac{L.C}{\mult_0(C)}
                          > (\text{resp}. \ge) \frac{a \sqrt[n]{L^n} (D.C)}{n \mult_0(C)},
\end{align}
so that
\begin{align}
1 > \frac {D.C}{\mult_0 C}.
\end{align}
Now for a contradiction, we assume that $C$ is not contained in $D$.
Then we obtain 
\begin{align}\label{eq:xy}
D.C \ge \mult_0(C) \mult_0 (D) = \mult_0(C),
\end{align}
a contradiction.
\end{proof}

\begin{lemma}\label{lm:main2}
Let $X$ be a smooth variety of dimension $n$, and $D$ be a divisor containing a point
$x \in X$.
Assume
\begin{align}
\sqrt[n]{L^n} \ge \sqrt[n-1]{a (L|_D)^{n-1} }
\end{align}
 for $a > \sqrt[n] {\frac {{n}^{n-1}} {\mult_x(D)}}  $.
 Then we have the upper bound
$\varepsilon (X,L;x) < a \sqrt[n]{L^n} / n $.
\end{lemma}
\begin{proof}
Suppose that 
$\varepsilon (X,L;x) \ge a \sqrt[n]{L^n} / n $.
Then we obtain the following inequality;
\begin{align}
 \frac{a \sqrt[n-1]{a (L|_D)^{n-1} }} {n} 
 \le   \frac{a \sqrt[n]{L^n}}{ n} 
 \le \varepsilon (X,L;x) 
\le \sqrt[n-1]{\frac {(L|_D)^{n-1}}{\mult_x (D)}}.
\end{align}
This contradicts to the assumption that $a > \sqrt[n] {\frac {{n}^{n-1}} {\mult_x(D)}}  $.
\end{proof}

\section{Applications}\label{sc:applications}
In this section, we give some applications of our theorems.
First, we show some results about the uniqueness of Seshadri curves by applying \pref{pr:main2}.

Let $(A,L)$ be a polarized abelian variety of dimension $n$.
For any $a \in \bR_{>0}$,
we denote the set of all curves satisfying
\begin{align}\label{eq:nounou}
\varepsilon_{C} (L)  < \frac{a \sqrt[n]{L^n}} {n}
\end{align}
 by $\cC_{a}$.
Moreover, we define 
\begin{align*}
\cD _{a} \coloneqq 
\{D \subset A \mid \text{D is an abelian divisor satisfying the following \pref{eq:moumou}}\}.
\end{align*}
\begin{align}\label{eq:moumou}
\sqrt[n]{L^n}  \ge \sqrt[n-1]{a (L|_D)^{n-1} }.
\end{align}

Then, by \pref{pr:main2}, it follows that
\begin{align}\label{eq:containing}
\bigcup_{C \in \cC_{a} } C \subset \bigcap _{D \in \cD _{a}} D.
\end{align}
This observation implies the following \pref{pr:submaximal2} and \pref{pr:sumaximal3}.

\begin{proposition}\label{pr:submaximal2}
Let $(S,L)$ be a polarized abelian surface.
Assume that there exists a curve $C \ni 0 $ such that
\begin{align}\label{eq:intro2dima}
  \varepsilon_C (L) < {\sqrt{\frac {L^2}{2}}} .
\end{align}
Then $C$ is elliptic and it is the unique curve satisfying \pref{eq:intro2dima} and containing $0 \in S$.
\end{proposition}
 \begin{proof}
 The definition of the Seshadri constant implies that 
 $ \varepsilon (S,L) < {\sqrt{\frac {L^2}{2}}} $.
 Then, by \pref{lm:surface} below, there exists an elliptic curve $C_0$ such that 
 $ \varepsilon_{C_0} (L) = \varepsilon (S,L) $.
 However, applying \pref{pr:main2} as $a = \sqrt{2}$, we conclude that $C_0$ is the unique curve satisfies \pref{eq:intro2dima}.
 \end{proof}

\begin{remark}\label{rm:submaximal}
If $\sqrt{L^2}$ is irrational,
it is already known that there are at most only a finitely many submaximal curves.
In fact, by the proof of \cite[Theorem A.1.(a)]{MR1660259},
there exists an integer $k >0$ and $D \in|kL|$ such that
any curve satisfying  $\varepsilon_{C} (L) < {\sqrt{L^2}} $ is 
an irreducible component of $D$ by \cite[Lemma 5.2]{MR1678549}.
However, \pref{pr:submaximal2} implies that the Seshadri curve is unique and elliptic under the assumption \pref{eq:intro2dima}.
\end{remark}

\begin{proposition}\label{pr:sumaximal3}
Let $(A,L)$ be a polarized abelian threefold. For any $a \in \bR_{>0}$,
if there exist at least two abelian divisors in $\cD _{a}$, then there is at most only one curve in $\cC_{a}$ and it is an elliptic curve.
Moreover, if one can take $a \ge ({\sqrt[3]{3})}^2$, 
there exists exactly one curve in $\cC_{a}$.
\end{proposition}

\begin{proof}
Let $D_1$ and $D_2$ be different abelian divisors in $\cD _{a}$.
Then $D_1 \cap D_2$ with induced reduced structure
is a reduced algebraic group of dimension one. 
Hence the identity component is an elliptic curve.
Therefore we obtain the assertion since any curve in $\cC_{a}$ is contained in
the identity component of $D_1 \cap D_2$ by \pref{eq:containing}.
For the latter part, it is sufficient to show $\cC_{a} \neq \emptyset$.
However, this follows from \pref{lm:main2}.
\end{proof}

Now let us prove \pref{cr:main3}. For the proof, we use the following fact from \cite[Theorem A.1.(b)]{MR1660259} for abelian surfaces.

\begin{lemma}[{$=$\cite[Theorem A.1.(b)]{MR1660259}}] \label{lm:surface}
Let $(S,L)$ be a polarized abelian surface.
Then one has a lower bound
\begin{align}
\varepsilon(S,L) \ge \min \biggl\{ \varepsilon_0,  \frac {\sqrt{14 L^2 }}{4}  \biggr\},
\end{align}
where $\varepsilon_0$ is the minimal degree of the elliptic curves in $S$ with respect to $L$.
\end{lemma}

Then \pref{th:low_seshadri}, \pref{th:main2} and \pref{lm:surface} imply the following corollary.

\begin{corollary}\label{cr:main1}
Let $(A,L)$ be a polarized abelian threefold.
Assume that $  \varepsilon (A,L) < \sqrt[3]{L^3} /3  $.
\begin{enumerate}
\item 
If there exists an abelian surface $S$ which satisfies
\begin{align}\label{eq:assumption2}
 \sqrt[3]{L^3} > \frac {3 \sqrt{14 (L|_S)^2}}{4},
\end{align}
then $\varepsilon (A,L) = \varepsilon (S,L| _S)$.

\item Otherwise, $\varepsilon (A,L)$ is computed by an elliptic curve.
\end{enumerate}
%
\end{corollary}

\begin{proof}
First, we prove $(2)$.
By \pref{th:low_seshadri}, we may assume that there exists an abelian surface $S'$ such that 
\begin{align}
 \varepsilon(S', L|_{S'}) = \varepsilon(A,L) < \frac{\sqrt[3]{L^3}}{3} 
 \le \frac{\sqrt{14(L|_{S'})^2}}{4}.
\end{align}
Hence, applying \pref{lm:surface}, $\varepsilon(A,L)$ is the minimal degree of the elliptic curves in $S'$ with respect to $L$.

On the other hand, the assumption of $(1)$ implies
\begin{align}
\sqrt[3]{L^3} > \frac {3 \sqrt{14 (L|_S)^2}}{4}.
\end{align}
Then we conclude the proof by \pref{th:main2}.
\end{proof}

Then \pref{cr:main3} follows easily from \pref{cr:main1}.

\begin{proof}[Proof of \pref{cr:main3}]
Note that an ample line bundle on an abelian surface has a positive and even degree.
Since $L^3 \le 174 $, if there exists an abelian surface $S$ such that
\begin{align}\label{eq:vv}
\sqrt[3]{L^3} > \frac {3 \sqrt{14 (L|_S) ^2 }}{4},
\end{align}
it must be $(L|_S )^2$ = 2.
Hence, by \pref{cr:main1}, it follows that either 
\begin{enumerate}
\item $\varepsilon(A,L)$ is computed by an elliptic curve $C$, or
\item $A$ contains a principally polarized abelian surface $(S, L|_S)$ and
$\varepsilon (A,L) = \varepsilon (S,L|_S)$.
\end{enumerate}
If the first case occurs, we obtain $C.L= \varepsilon(A,L) = 1 $.
Moreover, $A$ is isomorphic to $C \times B$ for some abelian surface $B$ by
\cite[Lemma 1]{Debarre:2007aa}.
In the second case, again by \cite[Lemma 1]{Debarre:2007aa},
it follows that there exists an elliptic curve $E$ such that
$A \simeq S \times E$.
Furthermore, 
$\varepsilon(A,L) = 1$ or $4/3$ since it is known that a principally polarized abelian surface 
is isomorphic to either
the Jacobian variety of a genus two curve with the Theta divisor
or
the product of two elliptic curves with the product of line bundles of degree one.
In the case $\varepsilon(A,L) = 4/3$, then $S$ is the Jacobian variety of a genus two curve and any curve satisfying $\varepsilon_C (L) < 21 \sqrt[3]{L^3} / 8$ is contained in $S$ by \pref{eq:vv} and \pref{pr:main2}.

Finally, note that $\varepsilon(A,L) = 1$ if $ L^3 \le 60$ since $\sqrt[3]{60}/3 < 3/4$.
\end{proof}

\begin{remark}
It is known that any polarized abelian surface $(S,L)$ satisfies $\varepsilon(S,L) \ge 4/3$ if $\varepsilon(S,L)$ is not one (see \cite[Theorem 1.2]{MR1393263}).
Hence, in fact, the second assertion in \pref{cr:main3} can be proven also directly 
from Nakamaye's \cite[Lemma $3.3$]{MR1393263} or \pref{th:low_seshadri}.
Indeed, if $L^3 \le 60$, then $ \varepsilon (A,L) < \sqrt[3]{L^3} /3 $ implies 
$
\varepsilon (A,L) < 4/3.
$
However, \pref{th:low_seshadri} implies that 
$ \varepsilon (B,L|_B) =\varepsilon (A,L) < 4/3 $
where $B$ is an abelian surface or an elliptic curve.
Hence $\varepsilon (A,L)$ must be one.
\end{remark}

In \pref{cr:main3}, the assumption $ L^3 \le 60$ is optimal for $\varepsilon(A,L)$ to be one.
In the following example, we construct polarized abelian threefolds $(A,L)$ satisfying 
$\varepsilon(A,L) < \sqrt[3]{L^3}/ 3$, $L^3 =n$, and $\varepsilon(A,L) \neq 1$ for any $n \in 6 \bZ $ such that $n > 60$.


\begin{example}\label{eg:3-dim}
Let $(J,\theta)$ be a pair of the Jacobian variety of a genus two curve and its Theta divisor.
Then it follows that $\varepsilon (J,\theta) = 4/3$ by \cite[Proposition 2]{MR1612681}.
Assume $(E,M_k)$ is a polarized elliptic curve with  $\deg M_k =k >0$.
Consider $A \coloneqq J \times E$ and the ample line bundle
$ L_k \coloneqq pr_1 ^* \theta \otimes pr_2 ^* M_k$ on it.
Then straightforward computation implies
$ 
L _k ^3 = 6 k.
$
Then we obtain 
\begin{align}
\varepsilon( A, L_k ) = \min \{ \varepsilon (J,\theta), \varepsilon(E,M_k) \} 
= \min \{ 4/3, k \}.
\end{align}
Hence, if $ 2 \le  k \le 10$, we obtain $\varepsilon(A ,L_k) = 4/3$ and $\varepsilon (A,L_k) > \sqrt[3]{L_k ^3} /3$.
If $ 11 \le k $, we have $\varepsilon(A ,L_k) = 4/3$, $L^3 = 6k > 60$ and 
$\varepsilon (A,L_k) < \sqrt[3]{L_k ^3} /3$.

\end{example}

\bibliographystyle{amsalpha}
\bibliography{bib}

\end{document}